\documentclass[11pt,reqno,a4paper]{amsart}
\usepackage{latexsym,amsmath}
\usepackage[utf8]{inputenc}
\usepackage{amsthm}
\usepackage{amssymb}
\usepackage{epsfig}
\usepackage{microtype} %
\usepackage{fullpage}
\usepackage{setspace}
\usepackage{hyperref}
\usepackage{url}
\usepackage[shortlabels]{enumitem} 
\usepackage[normalem]{ulem}

\usepackage[mathlines]{lineno}
\usepackage{etoolbox} %

\newcommand*\linenomathpatch[1]{%
   \expandafter\pretocmd\csname #1\endcsname {\linenomath}{}{}%
   \expandafter\pretocmd\csname #1*\endcsname{\linenomath}{}{}%
   \expandafter\apptocmd\csname end#1\endcsname {\endlinenomath}{}{}%
   \expandafter\apptocmd\csname end#1*\endcsname{\endlinenomath}{}{}%
 }
\newcommand*\linenomathpatchAMS[1]{%
    \expandafter\pretocmd\csname #1\endcsname {\linenomathAMS}{}{}%
    \expandafter\pretocmd\csname #1*\endcsname{\linenomathAMS}{}{}%
    \expandafter\apptocmd\csname end#1\endcsname {\endlinenomath}{}{}%
    \expandafter\apptocmd\csname end#1*\endcsname{\endlinenomath}{}{}%
}

\expandafter\ifx\linenomath\linenomathWithnumbers
\let\linenomathAMS\linenomathWithnumbers
\patchcmd\linenomathAMS{\advance\postdisplaypenalty\linenopenalty}{}{}{}
\else
\let\linenomathAMS\linenomathNonumbers
\fi

\linenomathpatchAMS{gather}
\linenomathpatchAMS{multline}
\linenomathpatchAMS{align}
\linenomathpatchAMS{alignat}
\linenomathpatchAMS{flalign}
\linenomathpatch{equation}

\DeclareMathOperator{\dcut}{\textsl{d}_{\square}}
\DeclareMathOperator{\Forb}{\textsl{Forb}}
\DeclareMathOperator{\Forbmon}{\textsl{Forb}_{mon}}
\def\Forbhom^*{\Forb_{\textsl{hom}}^*} 

\DeclareMathOperator{\WOWZER}{{\tt WOWZER}}

\DeclareMathOperator{\poly}{{\tt poly}}

\DeclareMathOperator{\dist}{\textsl{d}_1}
 
\let\hom\relax \DeclareMathOperator{\hom}{ind}

\def\d#1{d_{#1}}
\def\dapprox#1#2#3{d^{(#1,#2)}_#3}

\def\cPstar#1#2{\mathcal{P}^*_{(#1,#2)}}
\def\AlmostReducible#1#2{\cG_{#1}^{(#2)}}
\def\cPstar#1{\mathcal{P}^*_{(#1)}}
\def\AlmostReducible#1{\cG_{#1}}

\def\eps{\varepsilon} \def\phi{\varphi} \def\cP{\mathcal{P}}
\def\cV{\mathcal{V}}  
  
\def\cG{\mathcal{G}}  \def\cF{\mathcal{F}}
\def\RR{\mathbb{R}} \def\NN{\mathbb{N}}  
\def\PP{\mathbb{P}} \def\({\left(} \def\){\right)} \def\>{\rangle}
\def\<{\langle}
\let\epsilon\varepsilon

\def\Floor#1{\lfloor#1\rfloor}

\long\def\FULLVERSION#1{}

\newcommand{\quot}[2]{\mathchoice%
{\left.\raisebox{.1em}{$\displaystyle{#1}$}\kern-1pt/\raisebox{-.2em}{$\displaystyle{#2}$}\right.}
{\left.\raisebox{.1em}{${#1}$}\kern-1pt/\raisebox{-.2em}{${#2}$}\right.}
{\left.\raisebox{.1em}{$\scriptstyle{#1}$}\kern-1pt/\raisebox{-.2em}{$\scriptstyle{#2}$}\right.}
{\left.\raisebox{.1em}{$\scriptscriptstyle{#1}$}\kern-1pt/\raisebox{-.2em}{$\scriptscriptstyle{#2}$}\right.}
} 
\usepackage{datetime}
\shortdate
\yyyymmdddate
\settimeformat{ampmtime}
\date{\today, \currenttime}

\title{On the query complexity of estimating the distance to hereditary
graph properties}

\author[Hoppen]{Carlos Hoppen}
\address{Instituto de Matem\'atica e Estat\'istica, UFRGS, Avenida Bento
  Gon\c{c}alves, 9500, 91501-970 Porto Alegre, RS, Brazil {\rm(C. Hoppen)}} 
\email{choppen@ufrgs.br}

\author[Kohayakawa]{Yoshiharu Kohayakawa}
\address{Instituto de Matem\'atica e Estat\'istica, Universidade de
  S\~{a}o Paulo, Rua do Mat\~{a}o 1010, 05508-090 S\~{a}o Paulo,
  Brazil {\rm(Y. Kohayakawa and H. Stagni)}} 
\email{\{yoshi|stagni\}@ime.usp.br}

\author[Lang]{Richard Lang}
\address{Institut f\"ur Informatik, Universit\"at Heidelberg, Im Neuenheimer Feld 205, 69120 Heidelberg, Germany {\rm(R. Lang)}}
\email{lang@informatik.uni-heidelberg.de}

\author[Lefmann]{Hanno Lefmann} 
\address{Fakult\"at f\"ur Informatik, Technische
Universit\"at Chemnitz, Stra\ss{}e der Nationen 62, 09111 Chemnitz,
Germany {\rm(H. Lefmann)}}
\email{lefmann@informatik.tu-chemnitz.de}

\author[Stagni]{Henrique Stagni}

\thanks{%
  C.~Hoppen acknowledges the support of CNPq (308054/2018-0) and
  FAPERGS (19/2551-0001727-8). 
  C.~Hoppen and H.~Lefmann acknowledge the support of CAPES and DAAD
  via PROBRAL (CAPES Proc.~88881.143993/2017-01 and DAAD~57391132 and~57518130).
  Y.~Kohayakawa was partially supported by CNPq (311412/2018-1, 
  423833/2018-9) and FAPESP (2018/04876-1).
  H.~Stagni was supported by
  FAPESP (2015/15986-4, 2017/02263-0) and CNPq (141970/2015-4).
  R. Lang acknowledges support by Millennium Nucleus Information and
  Coordination in Networks ICM/FIC RC130003.
  This study was financed in part by CAPES, Coordenação de
  Aperfeiçoamento de Pessoal de Nível Superior, Brazil, Finance
  Code~001. 
  CNPq is the National Council for Scientific and Technological
  Development of Brazil.  FAPESP is the S\~ao Paulo Research
  Foundation. Some results in this paper appeared in preliminary form
  in the Proceedings of EUROCOMB~2017~\cite{hoppen_et_al:eurocomb}.}
  
\usepackage{lineno}
\newcommand*\patchAmsMathEnvironmentForLineno[1]{%
\expandafter\let\csname old#1\expandafter\endcsname\csname #1\endcsname
\expandafter\let\csname oldend#1\expandafter\endcsname\csname end#1\endcsname
\renewenvironment{#1}%
{\linenomath\csname old#1\endcsname}%
{\csname oldend#1\endcsname\endlinenomath}}%
\newcommand*\patchBothAmsMathEnvironmentsForLineno[1]{%
\patchAmsMathEnvironmentForLineno{#1}%
\patchAmsMathEnvironmentForLineno{#1*}}%
\AtBeginDocument{%
\patchBothAmsMathEnvironmentsForLineno{equation}%
\patchBothAmsMathEnvironmentsForLineno{align}%
\patchBothAmsMathEnvironmentsForLineno{flalign}%
\patchBothAmsMathEnvironmentsForLineno{alignat}%
\patchBothAmsMathEnvironmentsForLineno{gather}%
\patchBothAmsMathEnvironmentsForLineno{multline}%
}
\usepackage[usenames,dvipsnames]{color}
\usepackage[normalem]{ulem}

\begin{document} 
\onehalfspace
\allowdisplaybreaks[2]
\footskip=28pt
\newtheorem{theorem}{Theorem}[section]
\newtheorem{cor}[theorem]{Corollary}
\newtheorem{lemma}[theorem]{Lemma}
\newtheorem{fact}[theorem]{Fact}
\newtheorem{property}[theorem]{Property}
\newtheorem{corollary}[theorem]{Corollary}
\newtheorem{proposition}[theorem]{Proposition}
\newtheorem{claim}[theorem]{Claim}
\newtheorem{conjecture}[theorem]{Conjecture}
\newtheorem{definition}[theorem]{Definition}
\theoremstyle{definition}
\newtheorem{example}[theorem]{Example}
\newtheorem{remark}[theorem]{Remark}

\begin{abstract}
	Given a family of graphs $\cF$, we prove that the normalized edit distance of any given graph $\Gamma$ to being
    induced $\cF$-free is estimable with a query complexity that depends
    only on the bounds of the Frieze--Kannan Regularity Lemma and on a Removal Lemma for $\cF$.
\end{abstract}

\maketitle

\section{Introduction and main results}

Property testing is concerned with very fast (randomized) algorithms for approximate decisions, where the aim is to distinguish between graphs that satisfy a given property and graphs that are `far' from satisfying this property. Research in this area has achieved tremendous success since its systematic study was initiated by Goldreich, Goldwasser and Ron~\cite{Goldreich:1998:PTC:285055.285060}. We refer to the monograph of Goldreich~\cite{goldreich_2017} for a detailed account of the evolution of this area.

In this paper, we consider randomized algorithms that have the ability to
query whether any desired pair of vertices in the input graph is
adjacent or not. Let $\cG$ be the set of finite simple graphs, and let $\cG(V)$ be the set of such graphs with vertex set $V$. 
We shall consider subsets $\cP$ of~$\cG$ that are closed under isomorphism, which are called \emph{graph properties}. To avoid technicalities, we restrict ourselves to graph properties $\cP$ such that $\cP \cap
\cG(V) \neq \emptyset$ for $V \neq \emptyset$. 
This includes all \emph{monotone} and \emph{hereditary} graph properties that contain arbitrarily large graphs, which~are 
properties that are inherited by subgraphs and by induced subgraphs, respectively. 

Here, we shall focus on hereditary properties. It is easy to see that for every hereditary property $\mathcal{P}$ there is a graph family $\mathcal{F}$ such that $\mathcal{P}$ equals $\Forb(\mathcal{F})$, which is the set of all graphs that do not have induced copies of any $F \in \mathcal{F}$. For example, one can take $\mathcal{F}$ to be the family of all graphs not satisfying $\mathcal{P}$.

A graph property $\mathcal{P}$ is said to be \emph{testable} if, for every
$\eps>0$, there exist a positive integer $q_\cP=q_\cP(\eps)$, called the
\emph{query complexity}, and a randomized algorithm $\mathcal{T}_\mathcal{P}$,
called a \emph{tester}, which may perform at most $q_\cP$ queries in an arbitrary input
graph, satisfying the following property. For an $n$-vertex input graph
$\Gamma$, the algorithm $\mathcal{T}_\mathcal{P}$ outputs `ACCEPT' with probability at least 2/3 if $\Gamma \in \mathcal{P}$ and `REJECT' with probability at least 2/3 if no graph obtained from $\Gamma$ by the addition or removal of fewer than $\eps n^2$ edges satisfies $\mathcal{P}$.
Alon and Shapira~\cite{AS2008} proved that every hereditary graph
property is testable. Moreover, in joint work with Fischer and Newman~\cite{AFNS2009}, they found a combinatorial characterization of testable graph properties. Recently, such a characterization has also been obtained for uniform hypergraphs by Joos, Kim, K\"{u}hn and Osthus~\cite{JKKO2017}.

Property testing may be stated in terms of graph distances: given two graphs $\Gamma$ and
$\Gamma'$ on the same vertex set, that is $V(\Gamma) = V(\Gamma')=V$, we define the \emph{normalized edit distance} between
$\Gamma$ and $\Gamma'$ by {$\dist(\Gamma,\Gamma')=
\left|E(\Gamma) \triangle E(\Gamma')\right| /|V|^2$, where $E(\Gamma) \triangle
E(\Gamma')$} denotes the symmetric difference of their edge sets. If $\cP$ is a graph property,  let the distance between a graph $\Gamma$ and $\cP$ be
\begin{equation*}
\dist(\Gamma,\cP)=\min\{\dist(\Gamma,\Gamma')\colon
V(\Gamma')=V(\Gamma) \textrm{ and }\Gamma'\in \cP\}.  
\end{equation*} 
Thus a graph property $\cP$ is testable if
there is a tester whose query complexity is bounded by a function of $\epsilon$
that distinguishes with probability at least $2/3$ between the cases
$\dist(\Gamma,\cP)=0$ and $\dist(\Gamma,\cP) \geq \eps$. Even though
the definition of property testing is quite flexible regarding
testers, Goldreich and 
Trevisan~\cite[Theorem~2]{Goldreich:2003:TTR:897701.897703} 
have shown that it is sufficient to
consider \textit{canonical testers}, which randomly choose a subset
$X$ of vertices in $\Gamma$ and then verify whether the induced subgraph
$\Gamma[X]$ satisfies some related property $\mathcal{P'}$.

Similarly, a function $z \colon \cG\to\RR$ is called a \emph{graph parameter}  if it is
invariant under relabeling vertices. A graph parameter $z\colon
\cG\to\RR$ is \emph{estimable} (or \emph{testable}) if, for every
$\eps>0$ and every large enough graph $\Gamma$, with probability at
least $2/3$, the value of $z(\Gamma)$ can be approximated up to an
additive error of $\eps$ by an algorithm that only has access to a
subgraph of $\Gamma$ induced by a set of vertices of size $s_z=s_z(\eps)$, chosen uniformly at random. 
The query complexity of such an algorithm is $\binom{s_z}2$ and the size
$s_z$ is called
its \emph{sample complexity}.  Estimable parameters have been introduced by Fischer and Newman~\cite{FN07}.
Borgs, Chayes, Lov\'{a}sz, S\'{o}s and Vesztergombi~\cite{Borgs2008} 
later gave a complete characterization of estimable graph parameters which, in particular, also implies that the distance
from hereditary graph properties is estimable. However, their approach
does not provide explicit bounds on the sample complexity. With a
different strategy, Fischer and Newman~\cite{FN07} proved that the
distance to every testable property is estimable, providing
$\WOWZER$\footnote{This function is one level higher in the Wainer
  hierarchy than the tower function.}-type bounds for the sample
complexity. Among other things, Alon, Shapira and Sudakov improved
this strategy for monotone
properties~\cite[Corollary~1.2]{ASS09},\footnote{For hereditary
  properties, see the comment at the end of \S1.2 in~\cite{ASS09}.} but
the query complexity still depends on constants given by strong
versions of the Regularity Lemma.

Goldreich and
Trevisan~\cite[Proposition~D.2]{Goldreich:2003:TTR:897701.897703}
proved that if~$\cP=\Forb(\cF)$ is a testable hereditary property, then, given a
graph~$\Gamma$, one can test~$\cP$ by checking whether or
not~$\Gamma[X]$ belongs to~$\cP$ for a random set~$X\subset V(\Gamma)$
of bounded size.  It follows that a hereditary property $\Forb(\cF)$
being testable is equivalent to the existence of an \emph{induced
  Removal Lemma} for $\cF$.
Our results rely on the following version of the removal lemma, which
was proved by Alon and Shapira~\cite{AS2008}. 
Given graphs $F$ and $\Gamma$, an induced homomorphism from $F$ to $\Gamma$ is a mapping $\phi \colon V(F) \to V(\Gamma)$  that preserves adjacency and non-adjacency.
More precisely, $uv \in E(F)$ if and only of $\phi(u)\phi(v) \in E(\Gamma)$.
We denote by $\hom(F,\Gamma)$ the induced homomorphism density of $F$ in $\Gamma$, which is the number of homomorphisms from $F$ to $\Gamma$ divided by the number of mappings from $V(F)$ to $V(\Gamma)$.
In other words, $\hom(F,\Gamma)$ is the probability that  a random mapping~$\phi\colon V(F)\to V(\Gamma)$ is an induced homomorphism.

\begin{lemma}[Induced Removal Lemma]\label{lem:removal} 
	For every~$\eps > 0$ and every (possibly infinite) family~$\cF$ of graphs, there exist~$M = M(\eps,\cF)$, $\delta = \delta(\eps,\cF)> 0$ and~$n_0=n_0(\eps,\cF)$
	such that the following holds. If a graph~$\Gamma$ on~$n \geq n_0$ vertices
	satisfies~$\dist(\Gamma,\Forb(\cF))\geq \eps$, then there is a graph ~$F\in\cF$ with~$|V(F)|\leq
	M$ such that~$\hom(F,\Gamma)\geq \delta$.  %
\end{lemma}
The first induced Removal Lemma was proved using a strong version of the
Regularity Lemma and therefore had upper bounds on $1/\delta$ and
$n_0$ of size $\WOWZER(\poly(1/\eps))$~\cite{Alon2000}.
The best known upper bound on the induced Removal Lemma is due to Conlon and Fox~\cite{conlon12:_bound}, but is still of tower-type with a tower height depending on both epsilon and the property $\Forb(\cF)$. It is known that no such upper bound may hold uniformly for all hereditary properties~\cite{AS08}; however, the height of the tower is known to be polynomial in $1/\eps$ if $\mathcal{F}=\{F\}$. Alon and Shapira~\cite{alon_shapira_2006} characterized all graphs $F$, with the possible exception of $F \in \{P_4,C_4\}$, such that the induced Removal Lemma for $\mathcal{F}=\{F\}$ holds for $\delta(\eps,F)=\poly(\eps)$. The case $F=P_4$ was shown to satisfy this property by Alon and Fox~\cite{AF12}, while Gishboliner and Shapira~\cite{GS19} made progress in the case $F=C_4$.
The question of deciding which hereditary properties admit a Removal
Lemma that may be proven without the use of the Regularity Lemma was raised
by Goldreich~\cite{Goldreich2011}, and by  Alon and Fox~\cite{AF12}, among others, and is currently under research. 

The goal of this work is to relate hereditary parameter estimation
directly to the bounds of Removal Lemmas by avoiding Szemer\'{e}di's
Regularity Lemma.
In~\cite{CKL+17},
the current authors proved a similar result for monotone properties. To this end, the concept of \emph{recoverable} graph properties was introduced. Roughly speaking, for a function $f\colon[0,1]\to\NN$, a graph property $\cP$ is $f$-recoverable if every large graph $G\in\cP$ is $\eps$-close to admitting a partition $\cV$ of its vertex set into at most $f(\eps)$ classes that witnesses membership in $\cP$ (i.e. such that any graph that can be partitioned in the same way must be
in $\cP$). It was shown that every monotone graph property $\Forbmon(\cF)$\footnote{As for hereditary properties, it is well-known that every monotone graph property $\cP$ is the set $\Forbmon(\cF)$ of all graphs that do not contain a copy (not necessarily induced) of a graph in a family $\cF$.} 
is $f$-recoverable for some function $f$ that depends only on the bounds of a `weighted' graph Removal Lemma for the family $\cF$. This improved the required sample complexity for
estimating $\dist({}\cdot{},\Forb_{{\rm mon}}(\cF))$ for families $\cF$ that admit Removal Lemmas with better bounds. Owing to recent improvements by Fox~\cite{Fox2011} on the bounds for the Removal Lemma in the case of families $\cF=\{F\}$ containing a single graph, this resulted in a better upper bound for distance estimation to monotone properties of this type.

Our main result here is an analogue of this for hereditary properties.
For every graph property $\cP$, let $\d\cP$ be the graph parameter defined by $\d\cP (\Gamma)
= \dist(\Gamma,\cP)$ for graphs $\Gamma$, which determines the distance between a graph and $\cP$.

\def\sdistXremoval{s_{\ref{thm:hereditary-estimable}}}
\begin{theorem}\label{thm:hereditary-estimable}
	Let $\cP = \Forb(\cF)$ be any hereditary property. 
	Then the graph parameter $\d\cP$ is estimable with sample complexity
	$\sdistXremoval(\eps) =  \exp\{\poly(\delta^{-M^2}, M, \log n_0)\},$
	where~$M$, $\delta$ and $n_0$ are as in Lemma~\ref{lem:removal} 
	with input~$\eps/2$ and $\cP$.
\end{theorem}

Theorem~\ref{thm:hereditary-estimable} provides an upper bound on
the sample complexity of estimating the distance to a hereditary
property~$\Forb(\cF)$, which solely depends on the upper bounds for
the associated Removal Lemma. In particular, for families $\cF$ that admit a Removal Lemma with sample
complexity polynomial in $1/\eps$, our result states that the distance to~$\Forb(\cF)$ can be estimated with a sample complexity that is exponential in
a polynomial in $1/\eps$. Such families are currently actively sought after.
Recent findings include the family consisting of a path on three edges~\cite{AF12}, 
finite families containing a bipartite, a co-bipartite and a split
graph~\cite{GS16} and the family of induced cycles of length at least four~\cite{Verclos2019} 
(for~$\cF$ as in the last example, $\Forb(\cF)$ is the set of chordal graphs).
This is a substantial improvement over
previous approaches, like~\cite{FN07} and \cite{ASS09}, which rely on
Szemer\'{e}di's Regularity Lemma and therefore provide bounds which are at
least a tower of height polynomial in $1/\eps$. 

As briefly mentioned before, our approach in our previous
work~\cite{CKL+17} is based on a removal
lemma for weighted graphs.  Since we wished to arrive at a result
involving classical (unweighted) removal lemmas, it was necessary to
relate our weighted removal lemma with the classical removal lemma.
In this paper, so that we can use the classical (induced) removal
lemma and its bounds directly, we take an alternative approach: instead of
recoverable properties, we consider the notion of `attestable'
properties.
Roughly speaking, a graph property $\cP$ is $f$-\emph{attestable} if every large graph $G\in\cP$ is $\eps$-close to admitting a partition $\cV$ of its vertex set into at most $f(\eps)$ classes that witnesses closeness to $\cP$ (i.e. such that any graph that can be partitioned in the same way must be close to $\cP$). Recall that, in contrast, \emph{recoverable} refers to membership in ${\mathcal P}$. 
The proof of Theorem~\ref{thm:hereditary-estimable} consists of two steps.
First we prove that $f$-attestable properties are estimable with sample complexity polynomial in $f$ (see Theorem~\ref{thm:attestable->estimable}).
Then we show that hereditary properties are $f$-attestable, where $f$ is
exponential in the bound given by Lemma~\ref{lem:removal}
(see Theorem~\ref{thm:hereditary-attestable}).

\section{Attestable properties}

We denote the set of all complete weighted graphs
(with loops and edge weights between~$0$ and~$1$) by~$\cG^*$.
Let~$\cG^*(V)$ be the set of all weighted graphs on vertex set $V$. 
The distance between two weighted graphs $R,R'\in\cG^*(V)$ is given by
\[
	\dist(R,R')= \frac{1}{|V|^2} \sum_{(i,j) \in V^2} |R(i,j)-R'(i,j)|.
\]
Note that we can interpret an (ordinary) graph $\Gamma$ as a weighted graph with weights $\Gamma(u,v)=1$ whenever $uv \in E(G)$ and $\Gamma(u,v)=0$ otherwise. As with graphs, a \emph{weighted graph property} $\mathcal{P}^*$ is a subset of $\cG^*$ that is closed under (weight-preserving) isomorphisms. For a weighted graph $R \in\cG^*(V)$ and a weighted graph property $\mathcal{P}^*$, let
\[\dist(R,\cP^*)=\min\{\dist(R,R') \colon  R' \in \cG^*(V) \cap  \cP^*\}.\]

An \emph{equipartition} of a graph~$\Gamma$ is a partition~$\cV = \{V_i\}_{i=1}^k$ of its vertex set~$V(\Gamma)$ such that~$|V_i|\leq |V_j|+1$ for all~$1\leq i,j\leq k$. 
Since we will only consider equipartitions $\{V_i\}_{i=1}^k$ of 
graphs of size much larger than $k$, we will ignore divisibility issues 
and assume that every class has {exactly} $n/k$ vertices.

Given an equipartition $\cV = \{V_i\}_{i=1}^k$ of a graph $\Gamma$, we write $\Gamma(V_i,V_j)$ for the number of edges $v_iv_j$ with $v_i \in V_i$ and $v_j \in V_j$. We assume that edges $vw$ with both ends in $V_i$ are counted twice by $\Gamma(V_i,V_i)$, namely as $vw$ and as $wv$.
The \emph{reduced graph}~$\quot{\Gamma}{\cV}\in \cG^*$ of~$\Gamma$ by~$\cV$ is a weighted graph
with vertex set $[k]=\{ 1, \ldots , k\}$ and edge weights 
\[
	\quot{\Gamma}{\cV}(i,j) = \frac{\Gamma(V_i,V_j)}{|V_i||V_j|}
\] for all
$1\leq i,j\leq k$. 
As we will see, the reduced graph $\quot{\Gamma}{\cV}$ can provide some
information about $\Gamma$ if the number of classes of $\cV$ is large enough (but still small with respect to the order of $\Gamma$), especially if $\cV$ is a
regular partition (in the sense of Frieze-Kannan). Given a set $V$ and an integer $K \leq |V|$ we denote the set of all equipartitions of~$V$ into at most~$K$ classes by~$\Pi_K(V)$.
We also define the set 
\[
\quot{\Gamma}{\Pi_K} = \{\quot{\Gamma}{\cV}: \cV \in
\Pi_K(V(\Gamma))\}
\]
of all reduced graphs of~$\Gamma$ with vertex size at most~$K$.

The next theorem is a slight modification of Theorem~3.2~\cite{CKL+17}. It states that if a graph parameter $z\colon\cG \to\RR$ can be expressed as the optimal value of a certain optimization problem over
$\quot{\Gamma}{\Pi_K}$, then $z$ is estimable with sample complexity which is
polynomial in $K$ and in the reciprocal of the error parameter. 
\def\squotfunc{s_{\ref{thm:estim-opt-quotient}}}
\begin{theorem}[Theorem~3.2~\cite{CKL+17}]\label{thm:estim-opt-quotient} 
Let $z\colon \cG\to\RR$ be a graph
      parameter and suppose that there are a weighted graph parameter
$z^*\colon\cG^*\to\RR$, an integer $K\geq 1$ and a constant~$c\geq1$ such
      that
      \begin{enumerate}
\item %
        $z(\Gamma) = \min_{R\in\quot{\Gamma}{\Pi_K}} z^*(R)$ for every
        $\Gamma\in\cG$, and
\item \label{itm:continuous}
        $|z^{*}(R)-z^{*}(R')| \leq c\cdot\dist(R,R')$ for all weighted
        graphs $R,R'\in\cG^*$ on the same vertex set.
      \end{enumerate}
      Then~$z$ is estimable with sample
      complexity~$\squotfunc(\eps)=\poly(K,c/\eps)$.
\end{theorem}

The proof of Theorem~3.2 in~\cite{CKL+17} shows that the polynomial $\poly(K,c/\eps)$ in the statement does not depend on the graph parameters $z$ and $z^*$, but rather on a result about sampling\footnote{The proof of Theorem~3.2 in~\cite{CKL+17} is actually for a parameter $z$ given by the \emph{maximum} of $z^*(R)$ over $\quot{\Gamma}{\Pi_K}$, rather than the minimum. The current version follows immediately by applying it to the parameters $-z$ and $-z^*$.}. Since we will use Theorem~\ref{thm:estim-opt-quotient} as a black box, let us briefly sketch its proof. Given a graph $\Gamma$, let $\cV$ be a partition of $V(\Gamma)$ into at most $K$ classes for which $z(\Gamma) =  z^*(\quot{\Gamma}{\cV})$. Next, let $\overline{\Gamma}$ denote a subgraph of $\Gamma$ on $q$ vertices chosen uniformly at random. A recent result of Shapira and the fifth author~\cite{shapira19:_partit} guarantees that 
if $q\geq \poly(K,c/\eps)$ then with high probability there is a partition $\overline{\cV}$ of $V(\overline{\Gamma})$ into  $|\cV|$ classes such that $\dist(\quot{\Gamma}{\cV},\quot{\overline{\Gamma}}{\overline{\cV}}) <\eps/c$. It then follows quickly that $z(\overline{\Gamma}) \leq z(\Gamma)+\eps$. 
Similarly, we have with high probability that every reduced graph of $\overline{\Gamma}$ is close to a reduced graph of $\Gamma$ in terms of the edit distance.
A symmetric argument then yields $z(\Gamma) \leq z(\overline{\Gamma}) + \eps$, as desired.

Now, let $\cP$ be a hereditary graph property and suppose that we want to estimate
the parameter $\d\cP(\Gamma)$ for a graph~$\Gamma$. Our aim is to show that $\d\cP(\Gamma)$ can
be approximated by an optimization parameter over the set $\quot{\Gamma}{\Pi_K}$, for
some positive integer $K$. This will allow us to apply Theorem~\ref{thm:estim-opt-quotient}
to get sample complexity bounds for estimating $\d\cP$.

For every weighted graph $R\in \cG^*$, we define the property of being reducible to 
a weighted graph that is very close to $R$ by %
\begin{equation}
  \label{eq:1}
  \AlmostReducible{R} = \{\Gamma\in\cG : \text{there is an
    equipartition $\cV$ of $\Gamma$ of size $|V(R)|$, }
  \dist(\quot{\Gamma}{\cV}, R)\leq {2}/{|\cV|}\}.
\end{equation}
We could define~$\AlmostReducible{R}$ by requiring
that~$\dist(\quot{\Gamma}{\cV},R)$ should be~$0$. However, this definition would not be `robust'.  For instance,
$\AlmostReducible{R}$ would be empty whenever~$R$ contains an edge
of irrational weight.  In fact, even if~$R$ is a reduced graph coming
from a concrete graph~$G$, with this more restrictive definition, it
could be that~$\AlmostReducible{R}$ fails to contain graphs of
infinitely many orders, because of trivial divisibility issues, and
this would be a problem in our proofs.
The definition in~\eqref{eq:1} avoids such anomalies. 
    
For a  graph property $\cP$ and for every $\eps>0$, we define 
\[\cPstar\eps = \{R\in\cG^* \colon \d\cP(\Gamma)\leq \eps \text{ for
	all } \Gamma\in \AlmostReducible{R}\}.\]
In other words, $\cPstar\eps$ is the set of all reduced 
graphs $R$ that attest $\eps$-closeness to $\cP$, 
in the sense that if a graph $\Gamma$ admits a reduced graph
close to $R$, then $\Gamma$ must be $\eps$-close to $\cP$.
In the definition of an attestable property, the elements of $\cPstar\eps$ play the role of `forcing fingerprints', and we may test if a graph $\Gamma$ is close to $\cP$ by checking whether one of its equipartitions produces a fingerprint that is close to $S \in \cPstar\eps$.
This  is formalized as follows.
\begin{definition}\label{def:attestable}
  Given a function $f\colon (0,1]\to\NN$, we say that a graph property
  $\cP$ is \emph{$f$-attestable} if, for any $\eps >0$ and any graph
  $\Gamma\in\cP$ with $|V(\Gamma)| \geq f(\eps)^{3/2}$, there exists a
  reduced graph $R\in \quot{\Gamma}{\Pi_{f(\eps)}}$ of $\Gamma$ for
  which $R\in\cPstar\eps$.\footnote{We remark that the choice of the exponent $3/2$ is not crucial, but it is convenient for applying technical results such as Lemma~\ref{lem:distance}.}
\end{definition}

It is possible to connect attestable properties with parameter distances. For an integer $K>0$ and $\eps>0$, we define the graph parameter
$\dapprox{K}{\eps}{\cP}\colon \cG \to [0,1]$ such that
\[
\dapprox{K}{\eps}{\cP}(\Gamma)=\min_{R\in\quot{\Gamma}{\Pi_K}}
\dist(R,\cPstar\eps).
\]
So if $\cP$ is $f$-attestable, then by definition 
$\dapprox{K}{\eps}{\cP}(\Gamma)=0$ for $K = f(\eps)$ and all graphs $\Gamma \in \cP$ with $|V(\Gamma)| \geq K^{3/2}$. The next lemma shows that $\dapprox{K}{\eps}{\cP}$ is our desired optimization parameter.
\begin{lemma}\label{lem:distance}
	Let $\cP$ be an $f$-attestable graph property for a function~$f\colon
	(0,1]\to\RR$. 
	Fix $\eps>0$ and let~$K = f(\eps)$.
	Then every graph~$\Gamma\in\cG(V)$ with $|V|\geq K^{3/2}$ satisfies
	$
	\left| \d\cP(\Gamma) - \dapprox{K}{\eps}{\cP}(\Gamma) \right| \leq \eps.
	$
\end{lemma}

Before we prove Lemma~\ref{lem:distance} let us see how it implies the main result of this section:
\def\sthmone{s_{\ref{thm:attestable->estimable}}}
\begin{theorem}\label{thm:attestable->estimable}
	Let $\cP$ be an $f$-attestable graph property, for a function
	$f\colon(0,1]\to\RR$. Then, the graph parameter
	$\d\cP$ is estimable with sample complexity
	$\sthmone(\eps) = \squotfunc(\eps/2) = \poly(f(\eps/2), 1/\eps)$.
\end{theorem}

\begin{proof}
	Fix $\eps>0$ and let $K=f(\eps/2)$.  Consider a graph $\Gamma$ on at least $K^{3/2}$ vertices. 
By Lemma~\ref{lem:distance}, we have
	\[
	\left|\d\cP(\Gamma) - \dapprox{K}{\eps/2}{\cP}(\Gamma) \right| \leq  \frac{\eps}{2}.
	\]
	Define 
	$z^*(R) = \dist(R,\cPstar{\eps/2})$ and 
	note that $\dapprox{K}{\eps/2}{\cP}(\Gamma) = \min_{R\in\quot{\Gamma}{\Pi_K}} z^*(R)$.
	Moreover, $|z^*(R)-z^*(R')|\leq 1\cdot \dist(R,R')$ for all $R,R'\in\cG^*$. Therefore, 
	we may apply Theorem~\ref{thm:estim-opt-quotient} to conclude
        that $\dapprox{K}{\eps/2}{\cP}(\Gamma)$ may be approximated
        within error $\eps/2$ with probability at least $2/3$ by
        randomly  choosing a subgraph of $\Gamma$ of size
        $\squotfunc(\eps/2)=\poly(f(\eps/2),1/\eps)$. In
        particular, $\d\cP(\Gamma)$ is approximated within error
        $\eps$, as desired. 
\end{proof}

Now to the proof of Lemma~\ref{lem:distance}.
The direction $\dapprox{K}{\eps}{\cP}(\Gamma)\leq \d\cP(\Gamma) + \eps$ follows essentially from the definition of attestable properties. 
On the other hand, the direction $\d\cP(\Gamma) \leq \dapprox{K}{\eps}{\cP}(\Gamma)  + \eps$ does actually not require the property $\cP$ to be attestable and can be derived by constructing a graph reducible to an appropriate element of $\cPstar\eps$.

\begin{proof}[Proof of Lemma~\ref{lem:distance}]
Fix~$0<\eps<1$,~$K=f(\eps)$. Let $V = [n]$ with $n \geq K^{3/2}$.
We first show that~$\dapprox{K}{\eps}{\cP}(\Gamma)\leq \d\cP(\Gamma)$. 
Let~$G\in \cP$ be a graph
such that~$\dist(\Gamma,G) = \d\cP(\Gamma)$.  Since~$\cP$ is~$f$-attestable, we
can fix an~equipartition~$\cV=\{V_i\}_{i=1}^k$, with $k\leq K$, for which
$\quot{G}{\cV}\in\cPstar\eps$. In particular, we have
\begin{align*}
\dapprox{K}{\eps}{\cP}(\Gamma) &\leq \dist(\quot{\Gamma}{\cV}, \quot{G}{\cV}) 
= \frac1{k^2}\sum_{(i,j)\in [k]^2}
\frac{|\Gamma(V_i,V_j) - G(V_i,V_j)|}{|V_i||V_j|}\\
&\leq \frac{1}{n^2}\sum_{(i,j)\in[k]^2}
\sum_{\substack{u\in V_i\\v\in V_j}} |\Gamma(u,v)-G(u,v)| = \dist(\Gamma,G)= \d\cP(\Gamma).
\end{align*}

Next, we proceed to show that~$\d\cP(\Gamma) \leq
\dapprox{K}{\eps}{\cP}(\Gamma) + \eps$.
Let~$R\in\quot{\Gamma}{\Pi_K}$ and~$S\in\cPstar\eps$ be  such that
$\dist(R,S)=\dapprox{K}{\eps}{\cP}(\Gamma)$. Let~$k=|V(R)|$ and fix an equipartition $\cV =
\{V_i\}_{i=1}^k$ of~$\Gamma$ such that~$R = \quot{\Gamma}{\cV}$.
Let us construct a graph $G \in\AlmostReducible{S}$ by modifying  $\Gamma$ as follows.
For each $1 \leq i < j \leq k$ such that $R(i,j)>S(i,j)$, we remove
exactly~$\Floor{(R(i,j)-S(i,j))|V_i||V_j|}$ edges from~$\Gamma$ between $V_i$ and
$V_j$; if $S(i,j)>R(i,j)$, we add exactly~$\Floor{(S(i,j)-R(i,j))|V_i||V_j|}$ edges
between $V_i$ and $V_j$ to $\Gamma$.
Indeed we have $G \in\AlmostReducible{S}$ as
\begin{align*}
\dist(\quot{G}{\cV}, S) &= \frac1{k^2} \sum_{(i,j)\in[k]^2} |\quot{G}{\cV}(i,j)-S(i,j)|\\
& \leq \frac1{k^2} \left(k +\sum_{(i,j):i\not=j; R(i,j) > S(i,j)}\left| \frac{\Gamma(V_i,V_j)}{|V_i||V_j|}-  \frac{\Floor{(R(i,j)-S(i,j))|V_i||V_j|}}{|V_i||V_j|} - S(i,j)\right|\right) \\
& + \frac1{k^2} \left(k +\sum_{(i,j):i\not=j; R(i,j) < S(i,j)} \left|\frac{\Gamma(V_i,V_j)}{|V_i||V_j|}+  \frac{\Floor{(S(i,j)-R(i,j))|V_i||V_j|}}{|V_i||V_j|}  - S(i,j)\right|\right) \\ 
& \leq \frac1{k^2} \left(k +\sum_{(i,j)\in[k]^2}\frac{1}{|V_i||V_j| }\right)
= \frac{1}{k}+ \frac{k^2}{n^2}
\leq \frac{2}{k}.
\end{align*}
Moreover,
\begin{align*}
\dist(\Gamma,G) &\leq \frac1{n^2}\sum_{(i,j)\in [k]^2}
|R(i,j) - S(i,j)| |V_i||V_j| \\
&= \frac1{k^2} \sum_{(i,j)\in [k]^2} |R(i,j)-S(i,j)|\\
&= \dist(R,S).
\end{align*}
Since $S\in\cPstar\eps$, it follows 
that $\dist(G,\cP)\leq \eps$. 
Hence, by the triangle inequality, $\d\cP(\Gamma)\leq \dist(\Gamma,G)+\eps \leq
\dapprox{K}{\eps}{\cP}(\Gamma)+\eps$, as required. 
\end{proof}

\section{Hereditary properties are attestable}
Let~$\phi\colon V(F)\to V(R)$ be a function from the vertex set of a
graph~$F\in\cG$ to the vertex set of a weighted graph~$R\in\cG^*$. Using $\binom{S}{2}$ to denote the set of all two-element subsets of a set $S$, the
\emph{induced homomorphism weight}
is defined as 
\[
\hom_{\phi}(F,R) = \prod_{ij\in E(F)}R\big(\phi(i),\phi(j)\big)\prod_{ij \in
	\binom{V(F)}{2} \setminus E(F)} \Big(1-R\big(\phi(i),\phi(j)\big)\Big).
\]
When $\phi$ is injective, we can interpret $\hom_{\phi}(F,R)$ as the probability that $\phi$ is an induced homomorphism from $F$ to $H$, where $H\in\cG(V(R))$ is a graph in which 
each edge $ij\in\binom{V(R)}2$ is independently present with probability $R(i,j)$. 

The \emph{induced homomorphism density}~$\hom(F,R)$ of~$F\in \cG$ in~$R\in \cG^*$ is defined as the average
homomorphism weight over all mappings $\phi\colon V(F) \to V(R)$. 
So if $\Phi$ denotes the set of functions from  $V(F)$ to $V(R)$, then
\[
\hom(F,R) =  \frac{1}{|\Phi|} \sum_{\phi \in \Phi} \hom_{\phi}(F,R).
\]
Note that if $\Gamma$ is a graph, then $\hom(F,\Gamma)$ is the probability that 
a random mapping~$\phi:V(F)\to V(\Gamma)$ is an induced homomorphism from
$F$ to $\Gamma$.
In particular, if $\Gamma\in\Forb(\{F\})$, then
\begin{equation}\label{equ:homFG-bound}
	\hom(F,\Gamma) \leq \binom{|V(F)|}2\cdot\frac1{|V(\Gamma)|},
\end{equation}
since a random mapping from $V(F)$ to $V(\Gamma)$ is not injective with probability at most $\binom{|V(F)|}{2} / |V(\Gamma)|$.
The next result shows that if $\hom(F,\Gamma)$ is bounded away from zero, then so is $\hom(F,\quot{\Gamma}{\cV})$.

\def\Ainj{A_{\text{inj}}} 
\begin{lemma}\label{lem:quotdensity}
	Let $F$ and $\Gamma$ be graphs and $f=|V(F)|$.
	Then $\hom(F,\quot{\Gamma}{\cV}) \geq \hom(F,\Gamma)^{\binom{f}{2}}$ holds for  every equipartition 
	$\cV=\{V_i\}_{i=1}^k$ of~$\Gamma$.
\end{lemma}

\begin{proof}
	Suppose that $\cV=\{V_i\}_{i=1}^k$ is an equipartition of $\Gamma$.
	Let $\Phi=V(\Gamma)^{V(F)}$ be the set of all functions from $V(F)$ to $V(\Gamma)$.
	For  a mapping $\alpha\colon V(F) \to [k]$ we set
	\begin{equation}\label{def:phi_alpha}
	\Phi_{\alpha} = \{\phi\in \Phi : \phi(u)\in V_{\alpha(u)} \text{ for all  $u\in V(F)$}\}.
	\end{equation}
	Let $\phi \in \Phi$ be chosen uniformly at random.  
	For all mappings $\alpha\colon V(F)\rightarrow[k]$ and all edges
	$uv\in E(F)$, we have
	\[\PP(\hom_{\phi}(F,\Gamma)=1 \mid \phi\in \Phi_{\alpha})
	\leq \PP(\Gamma(\phi(u),\phi(v))=1 \mid \phi\in \Phi_{\alpha}) = \quot{\Gamma}{\cV}(\alpha(u),\alpha(v)),\]
        as $\quot{\Gamma}{\cV}(\alpha(u),\alpha(v))$ is the probability that
		$\Gamma(x,y)=1$ when $x \in V_{\alpha(u)}$ and $y \in V_{\alpha(v)}$ are chosen uniformly and independently at random.
	Analogously, we also have 
	\[\PP(\hom_{\phi}(F,\Gamma)=1 \mid \phi\in \Phi_{\alpha}) \leq  \PP(\Gamma(\phi(u),\phi(v))=0 \mid \phi\in \Phi_{\alpha}) = 1-\quot{\Gamma}{\cV}(\alpha(u),\alpha(v))\]
	for all mappings $\alpha\colon V(F)\rightarrow[k]$ and all non-edges $uv \in \binom{V(F)}{2} \setminus E(F)$.
	We can apply the last two inequalities to bound the induced homomorphism density
	$\hom(F,\quot{\Gamma}{\cV})$ from below
	as follows
	\begin{eqnarray*}
	&&\hom(F,\quot{\Gamma}{\cV}) \\
	&\geq& \sum_{\alpha\colon V(F)\rightarrow[k]} \frac{1}{k^f} \cdot
	\left(\prod_{uv\in E(F)} \quot{\Gamma}{\cV}(\alpha(u),\alpha(v))
	\prod_{uv \in\binom{V(F)}{2} \setminus E(F)} \(1-\quot{\Gamma}{\cV}(\alpha(u),\alpha(v))\) \right)\\
	&=&
	\sum_{\alpha\colon V(F)\rightarrow[k]}\PP(\phi\in\Phi_{\alpha}) \cdot
	\left(\prod_{uv\in E(F)} \quot{\Gamma}{\cV}(\alpha(u),\alpha(v))
	\prod_{uv \in \binom{V(F)}{2} \setminus E(F)} \(1-\quot{\Gamma}{\cV}(\alpha(u),\alpha(v))\) \right)\\
	&\geq& \sum_{\alpha\colon V(F)\rightarrow[k]} \PP(\phi\in \Phi_\alpha) \cdot
	\PP(\hom_{\phi}(F,\Gamma)=1 \mid \phi\in \Phi_{\alpha})^{\binom{f}{2}}.
	\end{eqnarray*}
 	Since $x\mapsto x^{\binom{f}2}$ is convex for every $x\geq 0$, we get that
	\begin{align*}
	\hom(F,\quot{\Gamma}{\cV})	
	&\geq \(\sum_{\alpha\colon V(F) \to [k]} \PP(\phi\in \Phi_\alpha) \cdot \PP(\hom_{\phi}(F,\Gamma)=1 \mid \phi\in \Phi_{\alpha})\)^{\binom{f}2}= \hom(F,\Gamma)^{\binom{f}{2}}
	\end{align*}
	as desired.
\end{proof}

Note that the converse of Lemma~\ref{lem:quotdensity} does not hold in
general. 
For instance, the complete bipartite graph $\Gamma = K_{n,n}$ satisfies 
$\hom(K_3, K_{n,n}) = 0$, but $\hom(K_3, \quot{K_{n,n}}{\cV})$ is close 
to~$1/8$ if $\cV$ is a random equipartition of large size. 
However, $\hom(F,\Gamma)$ and $\hom(F,\quot{\Gamma}{\cV})$ are known to be close, provided $\cV$ is a Frieze-Kannan-regular partition.
To make this precise we need to set up some notation.
We define the \emph{cut-distance} between two weighted graphs $R_1,R_2 \in \cG^*(V)$ to be
\[
\dcut(R_1,R_2)=\frac1{|V|^2} \max_{\alpha,\beta} \Big|\sum_{\substack{x\in V\\ y\in V}} \alpha(x)\cdot (R_1(x,y)-R_2(x,y))\cdot \beta(y)\Big|,
\]
where the maximum is over all functions $\alpha,\beta\colon V\to[0,1]$.\footnote{This is equivalent to the definition of cut-distance in~\cite[Theorem 8.10]{LovaszHomBook}.}
Note that, by the triangle inequality and the fact that~$\alpha$
and~$\beta$ are bounded by~$1$, we have
$\dcut(R_1,R_2) \leq d_1(R_1,R_2)$.
Moreover, we can bound the difference of the induced homomorphism densities of a graph $F$ in $R_1$ and $R_2$  in terms of their cut distance:
\begin{lemma}\label{lem:hom-dens-F-in-R1-R2-bound-cut-distance}
	Let $R_1,R_2 \in \cG^*(V)$ be weighted graphs and $F$ a graph on $f$ vertices.
	Then
	$|\hom(F,R_1)-\hom(F,{R_2})| \leq f^2 \dcut({R_1},{R_2})$.
\end{lemma}

In the proof of Lemma~\ref{lem:hom-dens-F-in-R1-R2-bound-cut-distance} we will use the following fact.
\begin{fact}\label{fact:cube}
	Let $a_1,\dots,a_t$ and $b_1,\dots,b_t$ be real numbers. Then 
	\begin{equation}
    \prod_{i=1}^t a_i - \prod_{i=1}^t b_i = \sum_{j=1}^t \prod_{i<j}a_i\cdot (a_j-b_j)
    \cdot \prod_{i>j}b_i .%
    \end{equation}
\end{fact}

\begin{proof}%
Observe that
\begin{eqnarray*}
 &&\sum_{j=1}^t \prod_{i<j}a_i\cdot (a_j-b_j)
    \cdot \prod_{i>j}b_i 
    = \sum_{j=1}^t \prod_{i\leq j}a_i
    \cdot \prod_{i>j}b_i -  \sum_{j=1}^t \prod_{i<j}a_i 
     \prod_{i\geq j}b_i  
     = \prod_{i=1}^t a_i - \prod_{i=1}^t b_i,
\end{eqnarray*}
which is the desired result.
\end{proof}

\begin{proof}[Proof of Lemma~\ref{lem:hom-dens-F-in-R1-R2-bound-cut-distance}]
	Let $R_1,R_2 \in \cG^*(V)$ be weighted graphs with $n=|V|$ and $F \in \cG$ be a graph with $V(F)=[f]$.
	For every pair $uv\in \binom{[f]}2$, any vertices $x,y\in V$ and $i\in \{1,2\}$, define 
	\[
	g^{(u,v)}_i(x,y) = \begin{cases} R_i(x,y) &\text{if $uv\in E(F)$} \\
	1-R_i(x,y) &\text{if $uv\notin E(F)$.} \\
	\end{cases}
	\]
	We have 
	\[n^f(\hom(F,R_1)-\hom(F,R_2)) = 
	\sum_{(x_1,\dots,x_f)} 
    \biggl(\kern-1.3ex
	\prod_{uv\in\binom{[f]}2} g_1^{(u,v)}(x_u,x_v) - 
	\prod_{uv\in\binom{[f]}2} g_2^{(u,v)}(x_u,x_v)
    \biggr),\]
	where the sum is over all $n^f$ sequences of length $f$ of vertices in $V$. 
	By considering an arbitrary linear ordering $<$ of the elements $uv\in \binom{[f]}2$, 
	we apply Fact~\ref{fact:cube} to get 
	\begin{align*}
	&n^f(\hom(F,R_1) - \hom(F,R_2)) \\
	&=\sum_{(x_1,\dots,x_f)} 
	\sum_{uv\in\binom{[f]}2} 
	\Big(\kern-1.1ex\prod_{ab<uv}g_1^{(a,b)}(x_a,x_b)\Big) \cdot
	(g_1^{(u,v)}(x_u,x_v) - g_2^{(u,v)}(x_u,x_v))\cdot
	\Big(\kern-1.1ex\prod_{ab>uv}g_2^{(a,b)}(x_a,x_b)\Big) \\
	&= 
	\sum_{uv\in\binom{[f]}2} 
	\sum_{\vec{x}}\sum_{x_u,x_v}
	\Big(\kern-1.1ex\prod_{ab<uv}g_1^{(a,b)}(x_a,x_b)\Big) \cdot
	(g_1^{(u,v)}(x_u,x_v) - g_2^{(u,v)}(x_u,x_v))\cdot
	\Big(\kern-1.1ex\prod_{ab>uv}g_2^{(a,b)}(x_a,x_b)\Big),
	\end{align*}
	where the sum $\sum_{\vec{x}}$ is over all sequences 
	$\vec{x} = (x_w)_{w\in [f]\setminus\{u,v\}}$
	of $f-2$ vertices in $V$, indexed by vertices $w\in[f]$, with $w\not=u,v$.
	
	Fix $uv\in \binom{[f]}2$ and a sequence $\vec{x}$ as above. Then there 
	must be functions $\alpha^{\vec{x}}, \beta^{\vec{x}}\colon V\to[0,1]$ for which 
	we can write
	\begin{equation*}\label{eq:prodpulauv}
	\prod_{ab<uv}g_1^{(a,b)}(x_a,x_b)\cdot \prod_{ab>uv}g_2^{(a,b)}(x_a,x_b) = \alpha^{\vec{x}}(x_u)\cdot \beta^{\vec{x}}(x_v),
	\end{equation*}
	since no term on the left side of the equation depends on both $x_u$ and $x_v$.
	More precisely, there are three types of factors in each term of the above sum. Those that depend on $u$ (i.e., where $a$ or $b$ equals $u$), those depending on $v$ (i.e., where $a$ or $b$ equals $v$) and those that do neither depend on $u$ nor $v$ (i.e., where neither $a$ nor $b$ equals $u$ or $v$).
	We can then group the factors that depend on $u$ as a product, which we call $\alpha^{\vec{x}}(x_u)$, and group the remaining factors as another product, which we call $\beta^{\vec{x}}(x_v)$.

	Hence, 
	\begin{eqnarray*}
	&& n^f|\hom(F,R_1)-\hom(F,R_2)| \\
	&\leq& 
	\left| 
	\sum_{uv\in\binom{[f]}2} 
	\sum_{\vec{x}} \sum_{x_u,x_v}
	\left(\alpha^{\vec{x}}(x_u) (g_1^{(u,v)}(x_u,x_v) - g_2^{(u,v)}(x_u,x_v)) \beta^{\vec{x}}(x_v)
	\right) \right|\\
	&\leq& 
	\left| 
	\sum_{uv\in E(F)} 
	\sum_{\vec{x}} \sum_{x_u,x_v}
	\left(\alpha^{\vec{x}}(x_u) (R_1(x_u,x_v)-R_2(x_u,x_v)) \beta^{\vec{x}}(x_v)\right)
	\right|
	\\
	& +&
	\left| 
	\sum_{uv\in \binom{V(F)}{2} \setminus E(F)}
	\sum_{\vec{x}} \sum_{x_u,x_v}
	\left(\alpha^{\vec{x}}(x_u) (R_2(x_u,x_v) - R_1(x_u,x_v)) \beta^{\vec{x}}(x_v) \right)
	\right|.
	\end{eqnarray*}
	
	By the definition of the cut-distance, the absolute value of
        each of the sums over $x_u,x_v$ can be bounded by
        $\dcut(R_1,R_2)n^2$.  Therefore
	\begin{eqnarray*}
	&&|\hom(F,R_1)-\hom(F,R_2)|  \\
	&\leq& \frac{1}{n^f} \left(|E(F)| \cdot n^{f-2} \cdot \dcut({R_1},R_2) \cdot n^2+ \left(\binom{f}{2}-|E(F)|\right)\cdot n^{f-2}\cdot \dcut(R_1,R_2\cdot)n^2\right) \\
	&\leq&  f^2 \dcut({R_1},R_2),
	\end{eqnarray*}
	as required.
	\end{proof}

For an equipartition $\cV=\{V_i\}_{i=1}^k$  of a graph $\Gamma\in\cG(V)$,
 we define the \emph{blown-up reduced graph} $\Gamma_{\cV}\in
 \cG^*(V)$ by setting, for every $1 \leq i \leq j \leq k$ and vertices
 $u\in V_i$ and $v\in V_j$, 
\[\Gamma_{\cV}(u,v) = \frac{\Gamma(V_i,V_j)}{|V_i||V_j|}.\]
 We remark that the blown-up reduced graph has approximately the
  same induced homomorphism density of small graphs
  as the reduced graph.  Indeed, let $\Gamma$ be a graph with vertex  set
  $V$ of size $n$ and consider an equipartition $\cV=\{V_i\}_{i=1}^k$
  of $V$. Let $R= \quot{\Gamma}{\cV}$ and let $F$ be a graph on $f$
  vertices. Denote the sets of functions $\phi\colon V(F) \to V$ and
  $\alpha \colon V(F) \to [k]$ by $\Phi$ and $\Psi$, respectively,
  and, given $\alpha \in \Psi$, let $\Phi_\alpha$ be the subset of
  $\Phi$ defined in~\eqref{def:phi_alpha}. By definition, given any
  $\phi \in \Phi_\alpha$, we have $\hom_{\phi}(F,\Gamma_{\cV})  =\hom_{\alpha}(F,R)$, so that
\begin{eqnarray}\label{equ:hom-density-blown-up-reduced-graph}
	&&|\hom(F,R)-\hom(F,\Gamma_{\cV})| = \left|\frac{1}{|\Psi|} \sum_{\alpha \in \Psi} \hom_{\alpha}(F,R) -\frac{1}{|\Phi|} \sum_{\phi \in \Phi} \hom_{\phi}(F,\Gamma_{\cV}) \right| \nonumber \\
	&&= \frac{1}{|\Phi|} \left| \sum_{\alpha \in \Psi} \left(\frac{|\Phi| \hom_{\alpha}(F,R)}{|\Psi|} - \sum_{\phi \in \Phi_{\alpha}} \hom_{\phi}(F,\Gamma_{\cV})\right) \right|\nonumber \\
	&&= \frac{1}{|\Phi|} \left| \sum_{\alpha \in \Psi} \sum_{\phi \in \Phi_{\alpha}} \left( \frac{|\Phi| }{|\Psi| |\Phi_\alpha|} \hom_{\alpha}(F,R) -  \hom_{\phi}(F,\Gamma_{\cV}) \right) \right|\\
	&&\leq  \frac{1}{|\Phi|} \left| \sum_{\alpha \in \Psi} \sum_{\phi \in \Phi_{\alpha}} \left(\hom_{\alpha}(F,R) - \hom_{\phi}(F,\Gamma_{\cV}) \right)\right| +  \sum_{\alpha \in \Psi} \sum_{\phi \in \Phi_{\alpha}} \left| \frac{1}{|\Psi||\Phi_{\alpha}|}- \frac{1}{|\Phi|}  \right| \nonumber\\
	&&= \sum_{\alpha \in \Psi} \sum_{\phi \in \Phi_{\alpha}}
           \left| \frac{1}{|\Psi||\Phi_{\alpha}|}- \frac{1}{|\Phi|}
           \right|\phantom{\leq}
           \smash{\stackrel{\text{(*)}}{\leq}}\;\,
           \max\left\{ \left|\frac{n^f}{k^f \prod_{i=1}^f |V_{\alpha(i)}|}-1\right| \colon \alpha \in \Psi \right\} \leq \frac{2kf}{n}.\nonumber 
	\end{eqnarray}	
Inequality~(*) can be derived as follows:
\begin{eqnarray*}
\sum_{\alpha \in \Psi} \sum_{\phi \in \Phi_{\alpha}} \left| \frac{1}{|\Psi||\Phi_{\alpha}|}- \frac{1}{|\Phi|}  \right|&=&\sum_{\alpha \in \Psi} \sum_{\phi \in \Phi_{\alpha}} \frac{\big| |\Phi|-|\Psi||\Phi_{\alpha}| \big|}{|\Psi||\Phi_{\alpha}||\Phi| } \\
&=&\sum_{\alpha \in \Psi} \sum_{\phi \in \Phi_{\alpha}}  \frac{\left|n^f-k^f \prod_{i=1}^f |V_{\alpha(i)}| \right|}{k^f n^f \prod_{i=1}^f |V_{\alpha(i)}|} \\
&\leq& n^f \max\left\{ \frac{\left|n^f-k^f\prod_{i=1}^f |V_{\alpha(i)}|\right|}{k^f n^f \prod_{i=1}^f |V_{\alpha(i)}|}  \colon \alpha \in \Psi \right\}\\
&=&\max\left\{ \left|\frac{n^f}{k^f \prod_{i=1}^f |V_{\alpha(i)}|}-1\right| \colon \alpha \in \Psi \right\}.
\end{eqnarray*}
Since the partition is equitable, we know that, for every $i \in [k]$, we have $n/k-1 \leq |V_i| \leq n/k+1$. In particular, by Bernoulli's inequality,
\begin{equation*}
  \prod_{i=1}^f |V_{\alpha(i)}| \geq \left(\frac{n}{k}-1\right)^f \geq
  \left(\frac{n}{k}\right)^f\left(1-\frac{fk}{n}\right),
\end{equation*}
and hence
\begin{eqnarray*}
\max\left\{ \left|\frac{n^f}{k^f \prod_{i=1}^f |V_{\alpha(i)}|}-1\right| \colon \alpha \in \Psi \right\}  &\leq&
 \left| \frac{n^f}{k^f\left(\frac{n}{k}\right)^f\left(1-\frac{fk}{n}\right)} -1 \right| \\
&=& \left| \frac{n^f}{ n^f - kfn^{f-1}} - 1 \right| \\
&=&  \frac{kfn^{f-1}}{ n^f - kfn^{f-1}} \leq \frac{2kf}{n}
\end{eqnarray*}
for $n \geq 2kf$.

The equipartition $\cV$ is $\gamma$-\emph{FK-regular} if $\dcut(\Gamma,\Gamma_{\cV}) \leq \gamma$.
The Frieze--Kannan Regularity Lemma~\cite{FK99} tells us
that every sufficiently large graph admits an FK-regular partition whose size only depends on the approximation parameter $\gamma$. We shall use the version of this result stated in~\cite[Lemma 9.3]{LovaszHomBook}.
\begin{lemma}[Frieze--Kannan Regularity Lemma]\label{lem:frieze-kannan-regularity} 
	For every~$\gamma > 0$ and every $k_0>0$, there is~$K=k_0 \cdot
	2^{\poly(1/\gamma)}$ such that every graph $\Gamma$ on $n \geq K$ vertices
	admits a~$\gamma$-FK-regular equipartition into $k$ classes, 
	where $k_0\leq k \leq K$.
\end{lemma}
Note that Lemma~\ref{lem:frieze-kannan-regularity} is close to best possible, since Conlon and Fox~\cite{conlon12:_bound} found graph instances where the
number of classes in any $\gamma$-FK-regular partition is at
least~$k\geq2^{1/(2^{60}\gamma^2)}$ (for a previous result, see
Lov\'{a}sz and Szegedy~\cite{LS07}).
Now we are ready to show that hereditary graph properties are attestable.
\begin{theorem}\label{thm:hereditary-attestable}
	For every family~$\cF$ of graphs, the property~$\Forb(\cF)$ is
	$f$-attestable for~$f(\eps) = 2^{\poly(\delta^{-M^2},M,\log n_0)}$,  
	where $\delta,~M$ and $n_0$ are as in Lemma~\ref{lem:removal}
	with input $\cF$ and $\eps$.
\end{theorem}
\begin{proof}
	Let $\delta,M$ and $n_0$ be as in Lemma~\ref{lem:removal} with 
	inputs $\cF$ and $\eps$. 
	Let $K$ be as in Lemma~\ref{lem:frieze-kannan-regularity} with input
	\begin{equation*}
		k_0 = \max\left\{n_0, \frac{2}{\delta},{\frac{4M^2}{\delta^{M^2}}}\right\} \qquad \text{and} \qquad \gamma = \frac{\delta^{M^2}}{8M^2}.
	\end{equation*}
	Note that $K = 2^{\poly(\delta^{-M^2},\,M,\,\log{n_0})}$.
	Let $G\in\Forb(\cF)$ be a graph with $n\geq K^{3/2}$ vertices. 
	We claim that if $\cV$ is a $\gamma$-FK-regular equipartition 
	of $G$ into $k_0\leq k\leq K$ classes, then 
	$R:=\quot{G}{\cV} \in\cPstar\eps$.
	This will prove the theorem for $f(\eps)=K$.
	
	Suppose by contradiction that there is a graph $H\in \AlmostReducible{R}$ such that
	$\dist(H,\Forb(\cF))>\eps$. Since $|V(H)|\geq k_0\geq n_0$, 
	Lemma~\ref{lem:removal} asserts there must be a graph $F\in\cF$, 
	with $|V(F)|\leq M$, for which $\hom(F,H)\geq \delta$.
	As $H \in \AlmostReducible{R}$, there is a partition $\cV'$ of $H$ into $k$ classes for which $\dist(\quot{H}{\cV'},R)\leq 2/k$.
	It follows from Lemma~\ref{lem:quotdensity} that 
	\[\hom(F,\quot{H}{\cV'}) \geq \delta^{M^2}.\]
	Hence, by Lemma~\ref{lem:hom-dens-F-in-R1-R2-bound-cut-distance} 
    \[\hom(F,R) \geq \delta^{M^2} - M^2\dcut(\quot{H}{\cV'},R)\geq
     \delta^{M^2} - M^2\dist(\quot{H}{\cV'},R)\geq
     \delta^{M^2}- \frac{2M^2}{k} \geq \frac{\delta^{M^2}}{2}.\]
	On the other hand, $R$ is the reduced graph of $G$ with respect to
	the $\gamma$-FK-regular partition $\cV$.
	So by the above, Lemma~\ref{lem:hom-dens-F-in-R1-R2-bound-cut-distance} together with~\eqref{equ:hom-density-blown-up-reduced-graph}
	implies that 
	\begin{align*}
		\hom(F,G) &\geq \hom(F,G_{\cV}) - M^2   \dcut({G},{G_{\cV}})
		\\&\geq \hom(F,R) -  \frac{2kM}{n} -  M^2 \gamma
		\\&\smash{\stackrel{\text{(**)}}{\geq}}\phantom{\geq} \frac{\delta^{M^2}}{2}   -  2M^2 \gamma 
		\\&\geq \frac{\delta^{M^2}}{4},
		\end{align*}
		where~(**) holds because  
		$n \geq K^{3/2} \geq K^{1/2}k,$ and  $K^{1/2} \geq \frac{2}{M\gamma}=\frac{16M}{\delta^{M^2}}$ by our choice of $K$.
	But this contradicts~\eqref{equ:homFG-bound}, which asserts that $\hom(F,G)$ is at most
	$\binom{M}2\frac1n \leq \frac{M^2}{2k_0} \leq \frac{\delta^{M^2}}{8}$.
\end{proof}
Note that Theorem~\ref{thm:hereditary-estimable} follows from
Theorem~\ref{thm:attestable->estimable}
and~\ref{thm:hereditary-attestable}.

\section*{Acknowledgement}
\label{sec:acknowledgements}

We are thankful to anonymous referees for their careful reading and
insightful comments. Their suggestions have led to significant
improvements in the presentation of our results.

\endgroup

\end{document}